\documentclass[12pt]{amsart}
\usepackage{amscd,amssymb}%,showkeys}
\usepackage[mathscr]{eucal}
\usepackage[english]{babel}
\input pstricks
\input pst-node

\usepackage{pstricks, pst-node}
\usepackage[all]{xy}

\usepackage{amsmath}
\usepackage[latin1]{inputenc}%

\newlength{\itemlaenge}

\newtheoremstyle{mytheorem}% name
  {}%      Space above, empty = `usual value'
  {}%      Space below
  {\slshape}% Body font
  {}%         Indent amount (empty = no indent, \parindent = para indent)
  {\scshape}% Thm head font
  {.}%        Punctuation after head
  { }%     Space after thm head: " " = normal interword space;
  {}% Thm head spec

\newtheoremstyle{mydefinition}% name
  {}%      Space above, empty = `usual value'
  {}%      Space below
  {\upshape}% Body font
  {}%         Indent amount (empty = no indent, \parindent = para indent)
  {\scshape}% Thm head font
  {.}%        Punctuation after thm head
  { }%     Space after thm head: " " = normal interword space;
  {}% Thm head spec

\theoremstyle{mytheorem}
\newtheorem{lemma}{Lemma}[section]
\newtheorem{prop}[lemma]{Proposition}
\newtheorem*{prop*}{Proposition}

\newtheorem{cor}[lemma]{Corollary}

\newtheorem{thm}[lemma]{Theorem}

\newtheorem*{thm*}{Theorem}
\theoremstyle{mydefinition}
\newtheorem{rem}[lemma]{Remark}
\newtheorem*{rem*}{Remark}

\newtheorem*{notation*}{Notation}

\newtheorem*{warning*}{Warning}

\newtheorem{defi}[lemma]{Definition}
\newtheorem*{defi*}{Definition}

\newtheorem{exo}[lemma]{Example}

\numberwithin{equation}{section}

\newcommand{\bqn}{\begin{equation*}}
\newcommand{\eqn}{\end{equation*}}
\newcommand{\bq}{\begin{equation}}
\newcommand{\eq}{\end{equation}}
\newcommand{\ba}{\begin{aligned}}
\newcommand{\ea}{\end{aligned}}
\newcommand{\be}{\begin{enumerate}}
\newcommand{\ee}{\end{enumerate}}

\newcommand{\thismonth}{\ifcase\month % case 0 --- impossible!
  \or January\or February\or March\or April\or May\or June%
  \or July\or August\or September\or October\or November%
  \or December\fi}

\newcommand{\PSL}{\operatorname{PSL}}

\newcommand{\DD}{{\mathbb D}}

\newcommand{\NN}{{\mathbb N}}

\newcommand{\RR}{{\mathbb R}}

\newcommand{\ZZ}{{\mathbb Z}}

\newcommand{\Cc}{{\mathcal C}}
\newcommand{\Dd}{{\mathcal D}}

\newcommand{\Ff}{{\mathcal F}}

\newcommand{\Nn}{{\mathcal N}}

\newcommand{\G}{\Gamma}

\def\h{{\rm H}}

\def\T{\operatorname{T}}

\def\one{\mathbf{1\kern-1.6mm 1}}

\def\homeo#1{\operatorname{Homeo}^+\!\left(#1\right)}

\def\diam{{\operatorname{diam}}}

\def\h2{{\operatorname{H_2}}}
\def\h1{{\operatorname{H_1}}}

\def\PSL{\operatorname{PSL}}

\def\cs{{\check S}}

\def\hhs{S''}

\def\kgb{\kappa_G^{\rm b}}

\def\to{\rightarrow}

\def\h{{\rm H}}

\renewcommand{\phi}{\varphi}

\def\No{N\raise4pt\hbox{\tiny o}\kern+.2em}
\def\no{n\raise4pt\hbox{\tiny o}\kern+.2em}

\newcommand{\homr}{\textup{Homeo}_{\ZZ}(\RR)}
\newcommand{\homrp}{\textup{Homeo}^+_{\ZZ}(\RR)}

\newcommand{\gp}{G^+}

\begin{document}

\title[Order preserving representations]{On order preserving representations}
\author[G.~Ben Simon]{G.~Ben Simon}
\email{gabib@braude.ac.il}
\address{Department of Mathematics, ORT Braude College, P.O. Box 78, Karmiel 2161002, Israel}
\author[M.~Burger]{M.~Burger}
\email{burger@math.ethz.ch}
\address{Department Mathematik, ETH Z\"urich, R\"amistrasse 101, CH-8092 Z\"urich, Switzerland}
\author[T.~Hartnick]{T.~Hartnick}
\email{hartnick@tx.technion.ac.il}
\address{Mathematics Department, Technion - Israel Institute of Technology, Haifa, 32000, Israel}
\author[A.~Iozzi]{A.~Iozzi}
\email{iozzi@math.ethz.ch}
\address{Department Mathematik, ETH Z\"urich, R\"amistrasse 101,
  CH-8092 Z\"urich, Switzerland}
\author[A.~Wienhard]{A.~Wienhard}
\email{wienhard@mathi.uni-heidelberg.de}
\address{Mathematisches Institut, Ruprecht-Karls-Universit\"at Heidelberg, 69120 Heidelberg, Germany\newline
HITS gGmbH, Heidelberg Institute for Theoretical Studies, Schloss-Wolfs\-brunnen\-weg 35, 69118 Heidelberg, Germany}
\thanks{M.~B. was partial supported by the Swiss National Science Foundation project  200020-144373; 
T.~H. was partial supported by the Swiss National Science Foundation project 2000021-127016/2; 
A.~I. was partial supported by the Swiss National Science Foundation projects 2000021-127016/2 and 200020-144373;
A.~W. was partially supported by the National Science Foundation under agreement No. DMS-1065919 and 0846408,  by the Sloan Foundation, by the Deutsche Forschungsgemeinschaft, by the European Research Council under ERC-Consolidator grant no.\ 614733, and by the Klaus-Tschira-Foundation. 
Support by the Institut Mittag-Leffler (Djursholm, Sweden) and by the Institute for Advanced Study (Princeton, NJ) is gratefully
acknowledged.}

%\keywords{ }
%\subjclass{ }

\date{\today}

\begin{abstract} 
In this article we introduce order preserving representations of fundamental groups of surfaces into Lie groups with bi-invariant orders. 
By relating order preserving representations to weakly maximal representations, introduced in \cite{BBHIW1}, we show 
that order preserving representations into Lie groups of Hermitian type are faithful with discrete image and that the set of order preserving 
representations is closed in the representation variety. 
For Lie groups of Hermitian type whose associated symmetric space is of tube type 
we give a geometric characterization of these representations in terms of the causal structure 
on the Shilov boundary. 
\end{abstract}
\maketitle
%
%
%\tableofcontents
%

%\begin{center}
%*** FIRST DRAFT. PROBABLY STILL FULL OF MISTAKES. PRECISE REFERENCES TO \cite{BBHIW1} NEED STILL TO BE INCLUDED. SUPERFLUOUS REFERENCES NEED TO GET REMOVED.***
%\end{center}

\section{Introduction}
Let $\Sigma$ be a compact oriented surface (possibly with boundary) of negative Euler characteristic and $\pi_1(\Sigma)$ its fundamental group.
When $\Sigma$ is closed, every discrete and faithful representations of $\pi_1(\Sigma)$ into $\PSL(2, \RR)$ is the holonomy representation of a hyperbolic structure on $\Sigma$ (possibly with reversed orientation).
%, the set of discrete and faithful representations of $\pi_1(\Sigma)$ into $\PSL(2, \RR)$ is closed and open in the representation variety. Every discrete and faithful representation is the holonomy representation of a hyperbolic structure on $\Sigma$ (possible with reversed orientation) \cite{Nielsen}. 
In particular, if $\Gamma_1 , \Gamma_2$ are the fundamental groups of closed surfaces $\Sigma_1$, $\Sigma_2$ 
respectively, any abstract isomorphism $\Gamma_1 \to \Gamma_2$ is {\em geometric}, i.e. induced by a homeomorphism $\Sigma_1 \to \Sigma_2$ (Nielsen's theorem).
%
%Over the last few decades there have been several discoveries of geometrically interesting open subsets in the variety of representations of $\pi_1(\Sigma)$ into a semi-simple Lie group.  
%These subsets are the set of Hitchin
%representations \cite{Hitchin, Goldman_Choi, Labourie_anosov, Guichard_convex, Guichard_Wienhard_convex}, of 
%positive representations \cite{Fock_Goncharov,Fock_Goncharov_convex}, of maximal representations
%\cite{Goldman_thesis, Goldman_82, Toledo_89, Hernandez,
%Burger_Iozzi_Wienhard_tol, Burger_Iozzi_Wienhard_anosov,
%Burger_Iozzi_Labourie_Wienhard, Burger_Iozzi_Wienhard_htt, Wienhard_mapping,
%Hartnick_Strubel, Gothen, Bradlow_GarciaPrada_Gothen_survey, Bradlow_GarciaPrada_Gothen,
%Bradlow_GarciaPrada_Gothen_sp4, GarciaPrada_Gothen_Mundet} and of Anosov representations
%\cite{Labourie_anosov, Guichard_Wienhard_anosov}.  
%Quite surprisingly all these representations even have a geometric realization as holonomy structures of a locally homogeneous geometric structures (not necessarily defined on $\Sigma$)\cite{Goldman_Choi, Guichard_Wienhard_convex, Baraglia_thesis, Guichard_Wienhard_anosov}.  

When $\Sigma$ has non-empty boundary, the situation is completely different. The fundamental group $\pi_1(\Sigma)$ is isomorphic to a free group, and there are many abstract isomorphisms of a free group, which are not induced by a homeomorphism of surfaces. Similarly, not every discrete and faithful representation  into $\PSL(2, \RR)$ is a {\em hyperbolization}, i.e. the holonomy representation of a complete hyperbolic structure on the interior of $\Sigma$. 
%
%The set of discrete and faithful representations is closed, but not open in the representation variety. Not every discrete and faithful representation is the holonomy representation of a hyperbolic structure on $\Sigma$, and if $\Gamma_1 , \Gamma_2$ are the fundamental groups of surfaces $\Sigma_1$, $\Sigma_2$ with non-empty boundary not every abstract isomorphism of $\Gamma_1 \to \Gamma_2$ is induced by a homeomorpshim $\Sigma_1 \to \Sigma_2$.
However, as explained in Theorem \ref{thm:thm1.1} and Corollary \ref{cor:cor1.2} below, we can associate to every compact oriented surface $\Sigma$ a partially bi-ordered group $(\Lambda_\Sigma, \leq_\Sigma)$ such that homeomorphisms of surfaces correspond precisely to isomorphisms of the associated ordered groups, and such that (oriented) hyperbolizations of $\Sigma$ correspond to {\em order preserving representations} into the universal covering of ${\rm PSL}_2(\RR)$ (in a sense made precise below).

This motivates the study of representations of $\pi_1(\Sigma)$ into other Lie groups $G$ which are {\em order preserving} in the sense that they induce order-preserving homomorphisms from $(\Lambda_\Sigma, \leq_\Sigma)$ into some partially bi-ordered central extension of $G$. When $G$ is a semisimple Lie group of Hermitian type we show that the set of order preserving representations forms a closed subset of the variety of representations of $\pi_1(\Sigma)$ into $G$ consisting entirely of discrete and faithful representations. This is achieved by relating order preserving representations to the set of {\em weakly maximal representations} with positive Toledo number introduced in \cite{BBHIW1}, and in turn provides a geometric characterization of weakly maximal representations. This characterization takes on a particular nice form when the Hermitian symmetric space associated to $G$ is of tube type. 

We now outline the content of this article in more details.

\subsection{Surface groups and orders}\label{sec:1.1}
Let $\Sigma$ be a compact oriented surface of negative Euler characteristic. We aim to define a group $\Lambda_\Sigma$, which depends functorially on $\pi_1(\Sigma)$ and carries a family of bi-invariant partial orders which reflect the geometry of $\Sigma$. 

If $\partial \Sigma \neq \emptyset$ we define $\Lambda_\Sigma$ simply as the group of based homotopy classes of homologically trivial loops in $\Sigma$, i.e.
\[
\Lambda_\Sigma := [\pi_1(\Sigma), \pi_1(\Sigma)].
\]
If $\Sigma$ is closed, we first form the central extension
\bq\label{eq:central extension}
\xymatrix{
0\ar[r]
&\ZZ\ar[r]
&\widehat\Gamma\ar[r]^-{p_\Sigma}
&\pi_1(\Sigma)\ar[r]
&0
}
\eq
corresponding to the positive generator of $\h^2(\pi_1(\Sigma),\ZZ)$ (which is well defined since $\Sigma$ is oriented) and then define
\[
\Lambda_\Sigma := [\widehat{\Gamma}, \widehat{\Gamma}].
\]
This definition ensures that every homomorphism $\rho: \pi_1(\Sigma) \to \PSL(2,\RR)$ lifts to a unique homomorphism $\widetilde{\rho}: \Lambda_\Sigma \to \widetilde{{\rm PSL}}_2(\RR)$ from $\Lambda_\Sigma$ into the universal  cover $\widetilde{{\rm PSL}}_2(\RR)$ of ${\rm PSL}_2(\RR)$.

To describe a family of orders on $\Lambda_\Sigma$ we observe that $\widetilde{{\rm PSL}}_2(\RR)$ is a subgroup of the group $\homrp$ of increasing homeomorphisms of the real line
commuting with integer translations, and that the latter group admits a family of bi-invariant orderings $(\leq_q)_{q \in \NN}$ defined by
\bqn
\varphi_1\leq_q\varphi_2\quad\text{ if }\quad \varphi_1=\varphi_2 \text{ or } \varphi_1(x)+q<\varphi_2(x)\text{ for all }x\in\RR.
\eqn
\begin{thm}\label{thm:InvariantOrder} Let $h$ be a complete hyperbolic structure on the interior $\Sigma^\circ$ of $\Sigma$ which is compatible with the orientation, let $\rho_h: \pi_1(\Sigma) \to \PSL(2,\RR)$ be the corresponding holonomy representation with lift $\widetilde{\rho}_h: \Lambda_\Sigma \to \widetilde{{\rm PSL}}_2(\RR) \subset \homrp$. Then the bi-invariant partial orders $\leq_{q,\Sigma}$ on $\Lambda_\Sigma$ given by
\bqn
\gamma_1\leq_{q,\Sigma}\gamma_2\quad\text{ if }\quad\widetilde\rho_h(\gamma_1)\leq_q\widetilde\rho_h(\gamma_2)
\eqn
are independent of the choice of complete hyperbolic structure $h$.
\end{thm}
The theorem implies that the ordered groups $(\Lambda_\Sigma, \leq_{q,\Sigma})$ are \emph{topological invariants} of the surface $\Sigma$, i.e. every orientation-preserving homeomorphism $f: \Sigma_1 \to \Sigma_2$ of surfaces induces an isomorphism $(\Lambda_{\Sigma_1}, \leq_{q,\Sigma_1}) \to (\Lambda_{\Sigma_2}, \leq_{q, \Sigma_2}) $ of ordered groups for every $q \in \mathbb N$. Reversion of the orientation on $\Sigma$ simply reverses the orders $\leq_{q, \Sigma}$.

In the sequel we will call a homomorphism between ordered groups \emph{strictly order preserving} if it is order preserving and no strictly positive element is sent to the identity. We also denote $\leq_{\Sigma} := \leq_{0, \Sigma}$. The following theorem and its corollary show that the ordered group $(\Lambda_\Sigma, \leq_{\Sigma})$ is a complete topological invariant of $\Sigma$, which moreover detects hyperbolic structures.

\begin{thm}\label{thm:thm1.1}  Let $\rho\colon\pi_1(\Sigma)\to\PSL(2,\RR)$ be a homomorphism with associated lift $\widetilde{\rho}: \Lambda_\Sigma \to \widetilde{{\rm PSL}}_2(\RR) \subset \homrp$.  Then the following are equivalent:
\be
\item $\rho = \rho_h$ is the holonomy representation of a complete hyperbolic structure $h$ on $\Sigma^\circ$ which is compatible with the orientation.
\item $\widetilde\rho\colon(\Lambda_\Sigma, \leq_{q,\Sigma}) \to (\widetilde\PSL(2,\RR), \leq_q)$
is strictly order preserving for some $q \in \mathbb N$.
\item $\widetilde\rho\colon(\Lambda_\Sigma, \leq_{q,\Sigma}) \to (\widetilde\PSL(2,\RR), \leq_q)$
is strictly order preserving for every $q \in \mathbb N$.
\ee
\end{thm}

\begin{cor}\label{cor:cor1.2} Let $\Sigma_i$, $i=1,2$ be connected oriented surfaces with negative Euler characteristic.
Then an isomorphism $i\colon\pi_1(\Sigma_1)\to\pi_1(\Sigma_2)$ is geometric, i.e. induced by an orientation preserving homeomorphism
$\Sigma_1\to\Sigma_2$, if and only if its lift $\widetilde{i}\colon(\Lambda_{\Sigma_1}, \leq_{\Sigma_1})\to(\Lambda_{\Sigma_2}, \leq_{\Sigma_2})$ is strictly order preserving.
\end{cor}
In the case where $\partial \Sigma_1 = \partial \Sigma_2 = \emptyset$ we know from Nielsen's theorem that \emph{every} isomorphism $i\colon\pi_1(\Sigma_1)\to\pi_1(\Sigma_2)$ is geometric (hence order preserving) up to orientation-reversal. Corollary \ref{cor:cor1.2} can be seen as a variant of Nielsen's theorem for surfaces with boundary.

\subsection{Orders, order preserving homomorphisms and Hermitian groups}\label{Subsec1.2}  Given a bi-invariant order on a group $G$, 
we let $G^+:=\{g\in\G\colon\,g\geq e\}$ denote the set of positive elements.  It is a conjugacy invariant submonoid of $G$ and satisfies
$G^+\cap(G^+)^{-1}=\{e\}$.

While for general partial orders on a group $G$ there may be many elements $g,h \in G$ which are incomparable (i.e. neither $g \leq h$ nor $h \leq g$), we will be particularly interested in the following class of partial orders, in which elements are at least comparable to high enough powers of positive elements: 
\begin{defi}\label{defi:archimedean}  An order on a group $G$ is {\em Archimedean} if for any $g>e$ and $h$ arbitrary, 
there is an integer $n\geq1$ with $g^n\geq h$.
\end{defi}

Archimedean orders occur often in conjunction with orders {\em sandwiched} by quasimorphisms. This classs of orders was introduced in \cite{BenSimon_Hartnick_Comm} and further studied in \cite{BenSimon_Hartnick_JOLT, BenSimon_Hartnick_quasitotal}. It can be defined as follows.

\begin{defi}  An order on a group $G$ is sandwiched by a homogeneous quasimorphism $f:G\to\RR$
if there is a $C>0$ with 
\bqn
f^{-1}([C,\infty))\subset G^+\,.
\eqn
\end{defi}

It was established in \cite{BenSimon_Hartnick_Comm} that certain ratios of quasimorphisms can be reconstructed from sandwiched orders. We use this observation to establish the following abstract result that will allow us to link order preserving representations and weakly maximal representations.

\begin{thm}\label{thm:thm1.5}  Assume that the groups $H$ and $G$ are equipped with Archimedean orders which are
sandwiched respectively by homogeneous quasimorphisms $f_H$ and $f_G$.

If $\rho\colon H\to G$ is a strictly order preserving homomorphism, then, for some $\lambda>0$,
\bqn
f_G\circ\rho=\lambda\,f_H\,.
\eqn
\end{thm}

Given any bi-invariant ordering on  a group $G$, the set of \emph{dominant elements} \cite{Eliashberg_Polterovich} is defined as
\bqn
G^{++}\:=\{g\in G^+\colon \text{ for every }h\in H,\text{ there is }n\geq1\text{ with }g^n\geq h\}.
\eqn
It turns out that $G^{++}$ is the set of strictly positive elements of an Archimedean order on $G$, called the {\em associated Archimedean order}. It has in general fewer positive elements than the original one.

When $G^+$ is sandwiched by a homogeneous quasimorphism $f_G$,
we have the simple characterization 
\bqn
G^{++}=\{g\in G^+\colon f_G(g)>0\}\,,
\eqn
which shows in particular that $G^{++}$ is not empty, and that it is still sandwiched by $f_G$.

Let now $G$ be a connected simple Lie group of Hermitian type (meaning that the associated symmetric space is Hermitian) and with finite center.
Then there is a canonical connected $\ZZ$-extension $\check G$ of $G$; it admits
continuous bi-invariant orderings and they have been classified 
\cite{Hilgert_Hofmann_Lawson, Hilgert_Olafsson, Olshanski, BenSimon_Hartnick_Comm}.
It also admits an essentially unique (continuous) homogeneous quasimorphism
$f_G$ and it is shown in \cite{BenSimon_Hartnick_Comm} that any continuous order is sandwiched by $f_G$.

As a consequence of this and the characterization of weakly maximal representations obtained
in \cite[Proposition~3.2]{BBHIW1} we obtain the folllowing:

\begin{thm}\label{thm:thm1.6}  Let $\rho\colon\pi_1(\Sigma)\to G$ be a representation, and let $\preceq$ be the Archimedean ordering associated with a continuous bi-invariant ordering on $\check G$.
The following are equivalent:
\be
\item  The representation $\rho$ is weakly maximal with positive Toledo invariant.
\item There exists $q \in \mathbb N_0$ such that the induced homomorphism
\bqn
\widetilde\rho: (\Lambda_\Sigma, \leq_{q, \Sigma}) \to(\check G, \preceq)
\eqn
is strictly order preserving.
\ee
\end{thm}

\begin{rem} 
\begin{enumerate}
\item Similarly, weakly maximal representations with negative Toledo invariant correspond to strictrly order reversing representations.
\item Recall from \cite{BBHIW1} that weakly maximal representations with nonzero Toledo invariant form a closed subset of the representation variety which consists entirely of discrete faithful representations. In view of Theorem \ref{thm:thm1.6}, the same is true for the class of order-preserving representations defined by (2) (and the corresponding class of order-reversing representations).
\item While continuous orderings on $\check G$ have been described rather explicitly in 
\cite{Olshanski}, what is lacking is an explicit description of the associated set of dominant elements (or, equivalently, the associated Archimedean orders).
\item The theorem remains true if the continuous order under consideration is replaced by any of its perturbations in the sense of Lemma \ref{Perturbation}, see Remark \ref{PerturbationStable}.
\end{enumerate}
\end{rem}

\subsection{Causal orderings}\label{sec:causal_intro}
In the special case when $G$ is a simple Lie group of Hermitian type whose associated bounded symmetric domain $\Dd$ is of tube type, a particularly nice bi-invariant order on $G$ can be defined 
using the causal structure on the Shilov boundary of $\Dd$. In this case the Shilov boundary $\cs$ has infinite cyclic fundamental group, and the action of $G$ on $\cs$ lifts to an effective action of $\check G$ on its universal cover 
$\check R$. 
By a classical result of Kaneyuki  \cite{Kaneyuki} there exists an (essentially) unique  $G$-invariant {\em causal structure} $\mathcal C$ on $\check S$, i.e. a 
family of closed proper convex cones with non-empty interior $\mathcal C_x \subset T_x \cs$, such that $g_* \mathcal C_x = \mathcal C_{gx}$ for all $g \in G$. 
%(For more information on the classification of invariant causal structures on general homogeneous spaces see \cite{Vinberg, Olshanski, Hilgert_Hofmann_Lawson, Hilgert_Olafsson}.)

The causal structure $ \mathcal C$ on $\cs$ lifts to a $\check G$-invariant causal structure $\widetilde{\mathcal C}$
on $\check R$, which is (up to taking $- \widetilde{\mathcal C}$) in fact the only $\check G$-invariant causal structures on $\check R$.  
This causal structure can be used to make the following definition. For $x,y \in \check R$ we say $x\leq y$ if there exists a causal curve (see Definition~\ref{defi:causal curve}) in $\check R$ from $x$ to $y$.  We write $x<y$ if $x \leq y$ and  $x\neq y$.

It turns out that $\leq$ defines a partial order on $\check R$ (see Lemma~\ref{lem:orderinR}), and hence gives rise to a bi-invariant partial order on $\check G$ by setting
\bqn
g \preceq h\,\,\text{ if and only if }\,\, gx \leq hx\,\, 
\text { for all }x\in\check R\,.
\eqn
We refer to this order as the \emph{causal order} on $\check G$ and to the corresponding Archimedean order as the \emph{causal Archimedean order}. 
The dominants of these orders can be described as follows:

\begin{thm}\label{thm:thm1.8} The causal order on $\check G$ is sandwiched by the quasimorphism $f_{\check G}$ and its dominant elements admit the following description:
\bqn
G^{++}=\{g\in\check G\colon gx>x\text{ for all }x\in\check R\}\,.
\eqn
\end{thm}

Thus, unlike the case of a general Hermitian group, 
we have in this case a quite precise description of the set of dominants. This allows us to deduce:

\begin{thm}\label{thm:thm1.9}  Let $G$ be a simple connected Lie group with finite center,
whose associated symmetric space is of tube type.  Then there exists $q=q(G)\in\NN$
such that the following are equivalent for a representation $\rho: \pi_1(\Sigma) \to G$.
\begin{enumerate}
\item $\rho$ is weakly maximal with positive Toledo invariant.
\item If $\widetilde\rho\colon\Lambda_\Sigma\to \check G$ denotes the lift of $\rho$ and $\gamma \in \Lambda_\Sigma$ satisfies
$\gamma>_{q,\Sigma} e$, then 
\[
\widetilde\rho(\gamma) x>x \text{ for all } x \in \check R.
\]
\end{enumerate}
\end{thm}

\subsection{Organization of the article} This article is organized as follows: In Section \ref{Sec:Prelim} we collect various preliminary results concerning partial orders and their relations to quasimorphisms. The main result, Corollary \ref{cor:sandwich} ensures that if two quasimorphism sandwich the same partial order, then they are positive multiples of each other. Section \ref{SecGeneral} is the core section of the article. After a preliminary subsection concerning translation numbers we establish invariance of the ordered groups $(\Lambda_\Sigma, \leq_{q, \Sigma})$ (i.e. Theorem \ref{thm:InvariantOrder}) in Subsection \ref{SubsecOrderSurface}. The remaining results from Subsections \ref{sec:1.1} and \ref{Subsec1.2} are proved in Subsection \ref{SecMeta}. In fact, Theorem \ref{thm:thm1.6} is a consequence of a more general meta-theorem stated as Theorem \ref{thm:thm3.6}. Finally, Section \ref{sec:causal_order} is devoted to the study of the causal order and the proof of Theorems \ref{thm:thm1.8} and \ref{thm:thm1.9}. 

\section{Orders and quasimorphisms}\label{Sec:Prelim} 
\subsection{Partial orders and Archimedean orders }\label{Sec:dorder}
In this subsection we review some facts about partially ordered groups and set up some notation

\begin{defi}\label{BasicDefsOrders} Let G be a group.
A {\em bi-invariant (partial) order} on $G$ is a relation $\preceq$ on $G\times G$ such that the following hold for all $g,h,k \in G$:
\be
\item $g \preceq g$;
\item $g\preceq h$ and $h\preceq g$ implies that $g=h$;
\item $g\preceq h$ and $h\preceq k$ implies that $g\preceq k$, and
\item if $g\preceq h$ then $agb\preceq ahb$ for all $a,b\in G$.
\ee
%if $g \preceq h$ implies both $gk\preceq hk$ and $kg\preceq kh$ for all $g,h,k \in H$. In this case we refer to the pair $(H, \preceq)$ as a \emph{partially bi-ordered group}, and we write $g\prec h$ to mean $g \preceq h$ and $g\neq h$.
The {\em order semigroup} $G^+$ is defined as
\bqn
\gp=\{g\in G\colon\,g\succeq e\},
\eqn 
\end{defi}
Since $g \preceq h$ if and only if $g^{-1}h \in \gp$, the order $\preceq$ is uniquely determined by $\gp$. Order semigroups of bi-invariant partial orders are precisely the conjugation-invariant pointed submonoids of $G$ (where pointed means that $\gp\cap(\gp)^{-1}=(e)$), cf. \cite{BenSimon_Hartnick_Comm}.

 In \cite{Eliashberg_Polterovich} Eliashberg and Polterovich introduced the following subset of $\gp$:
\bqn
G^{++}:=\{g\in G^+\colon\text{ for every }h\in G\text{ there exists }p\geq1\text{ with }g^p\succeq h\}\,,
\eqn
whose elements they referred to as \emph{dominant elements} of the underlying order.
\begin{lemma}\label{lem:properties}  
The set $G^{++}$ is a strictly pointed, conjugation-invariant two-sided ideal in $G^+$, i.e.
\be
\item $G^{++}\cap(G^{++})^{-1}=\emptyset$;
\item $gG^{++}g^{-1} = G^{++}$ for all $g \in G$;
\item $G^{++}\cdot G^+=G^+\cdot G^{++}=G^{++}$.
\ee
In particular, $\{e\} \cup G^{++}$ is an order semigroup contained in $G^+$.
\end{lemma}

\begin{proof}  Properties (1) and (2) are clear.  It follows from bi-invariance,  that
for every $h\in G^+$ and $p\geq1$,
\bqn
(gh)^p=gh(gh)^{p-1}\succeq g(gh)^{p-1}\,,
\eqn
and thus, by recurrence, $(gh)^p\succeq g^p$.  
This implies that $G^{++}\cdot G^+=G^{++}$, and the other identity is proved similarly.
\end{proof}
In view of the lemma we refer to $G^{++}$ as the \emph{dominant semigroup} associated with $\preceq$. 

It may happen that the dominant semigroup of a bi-invariant partial order is empty. It may also happen that two different bi-invariant partial orders have the same dominant semigroup. However, it follows from the lemma that if $G^{++}$ is the dominant set of some partial order $G^+$, then there is always a unique minimal order semigroup with dominant semigroup $G^{++}$, and this minimal order semigroup is given by $\{e\} \cup G^{++}$. 
Here is a reformulation of our Definition~\ref{defi:archimedean}:

\begin{defi}\label{defi:Archimedean_order}
A bi-invariant partial order on $G$ is a {\em Archimedean} if the order semigroup $\gp$ equals  $\{e\} \cup G^{++}$.  
\end{defi}

One main advantage of dominant semigroups as opposed to order semigroups is their functoriality. Namely, if $\rho\colon G \to H$ is a group homomorphism and $H^+<H$ is an order semigroup, then the pre-image $G^+ := \rho^{-1}(H^+)$ need not be an order semigroup. Indeed, the kernel of $\rho$ is contained in $G^+$, hence if $\rho$ is not injective then $G^+$ is not pointed. On the contrary we observe:
\begin{lemma}\label{dominantsFunctorial} If $H^{++}$ is the dominant semigroup of a bi-invariant partial order on a group $H$ and $\rho\colon G \to H$ a group homomorphism, then $G^{++} := \rho^{-1}(H^{++})$ is also a dominant semigroup.
\end{lemma}
\begin{proof} Let $G^+ := \{e\} \cup G^{++}$. Then $G^+$ is a conjugation-invariant monoid, since the pullback of a conjugation-invariant semigroup is a conjugation invariant semigroup. To see that $G^+$ is pointed it suffices to observe that  \[G^{++} \cap (G^{++})^{-1} = \rho^{-1}(H^{++}\cap (H^{++})^{-1})=\emptyset.\]
Thus $G^+$ is an order semigroup and it remains to show that every non-trivial element in $G^+$ is a dominant. However, if $g \in G^{++}$ then $h:= \rho(g) \in H^{++}$. Thus for every $k \in G$ there exists $n \in \mathbb N$ such that $h^n \geq \rho(k)$, which implies that $\rho(g^nk^{-1}) \in H^{++}$ and thus $g^nk^{-1}\in G^{++}$, which is the desired dominance condition.
\end{proof}

%In practice it is not always easy to check whether a given partial order is a dominant order. We will often start from a general bi-invariant partial order and then pass to the associated dominant order given by the order semigroup $\{e\} \cup G^{++}$. Identifying this order explicitly is often difficult, see e.g the proof of Theorem \ref{thm:7.2}.
Given a bi-invariant partial order on a group $H$ we might consider a perturbation of this order. 
\begin{lemma}\label{Perturbation}
If $\preceq$ is a bi-invariant partial order on $H$ with order semigroup $H^+$, then also 
\[
H^+_g := \{e\} \cup \bigcap_{h \in H}h(g H^+)h^{-1} < H^+
\]
is an order semigroup for every $g\in H^+$.
\end{lemma}
\begin{proof} Since $H^+$ is a pointed semigroup and $g\in H^+$, also $\{e\} \cup gH^+$ is a pointed semigroup. The conjugate of a pointed semigroup is a pointed semigroup, and the intersection of pointed semigroups is a pointed semigroup.  It thus follows that $H^+_g$ is a pointed semigroup. Moreover, $H^+_g$ is conjugation-invariant by construction, hence an order semigroup. %Assume now that $H^+ = \{e\} \cup H^{++}$. Then $H^{++}_g := H^{+}_g \setminus\{e\} \subset H^{++}$. Thus for every $h \in H^{++}_g$ and every $k \in H$ there exists $n>0$ with $h^nk^{-1} \in H^{++}$.
\end{proof}
We call $H^+_g$ a \emph{perturbation} of $H^+$. 

\begin{rem}  The order $\leq_q$ on $\homr$ is a perturbation of $\leq$, where $g$ is the translation $x\mapsto x+q$.
\end{rem}

\subsection{Quasimorphisms and sandwiched orders}\label{Sec:qmsw}

In this section we present the basic facts and definition concerning
some bi-invariant orders on groups, which are ``sandwiched"
by a homogeneous quasimorphism.  The main point is a theorem
that reconstructs the quasimorphism (up to a multiplicative constant)
from the knowledge of the order.  This theorem was established in 
\cite{BenSimon_Hartnick_Comm, BenSimon_Hartnick_JOLT}. Since we later make use of some of the constructions, 
we give here a simpler account of this result.

%
%In this subsection we present the basic facts and definition concerning
%a certain class of bi-invariant orders on groups, namely orders ``sandwiched"
%by a homogeneous quasimorphism. The main point is a theorem
%that reconstructs the quasimorphism (up to a multiplicative constant)
%from the knowledge of the order. This material is taken from
%\cite{BenSimon_Hartnick_Comm, BenSimon_Hartnick_JOLT},
%where it is presented in a more elaborate form. 

Given a real-valued function $f\colon G\to \RR$ we denote by $df\colon G^2 \to \RR$ the function
\[
df(g,h) := f(h) - f(gh) + f(g).
\]
A function $f$ is a \emph{quasimorphism} if $df$ is bounded, in which case the supremum 
\bqn
\|df\|_\infty:=\sup_{g,h\in G}|f(gh)-f(g)-f(h)|\,.
\eqn
is called the \emph{defect} of $f$. A quasimorphism is called \emph{homogeneous} if $f(g^n) = n \cdot f(g)$ for all $g \in G$ and $n \in \ZZ$. Every homogeneous quasimorphism is automatically invariant under conjugation.

Given a quasimorphism $f$, the function $df$ can be seen as a $2$-cocycle in the inhomogeneous bar resolution of the bounded cohomology of $G$, and thus gives rise to a class $[df]$ in the second bounded cohomology $H^2_b(G; \RR)$ of $G$. If $G$ is locally compact and $f$ is Borel, then $f$ is automatically continuous and thus gives rise to a class $[df]$ in the second \emph{continuous} bounded cohomology  $H^2_{cb}(G; \RR)$ of $G$ (see \cite{Burger_Iozzi_Wienhard_tol}). We will use this relation to bounded cohomology later on. For now, let us explain how homogeneous quasimorphisms can be used to construct order semigroups.  

\begin{exo}\label{ExQm}
Let $f\colon G\to\RR$ be a homogeneous quasimorphism with defect 
$\|df\|_\infty$. For every $C\in\RR$, let us define
\bqn
\Nn_C(f):=\{g\in G\colon\,f(g)\geq C\}\,.
\eqn
Then for every $C\geq\|df\|_\infty$, 
the set $\Nn_C(f)\cup\{e\}$ is a conjugacy invariant
submonoid. If moreover $C>0$, then this semigroup is pointed, hence an order semigroup. 
\end{exo}

We are interested in bi-invariant partial orders whose order semigroup can be approximated by sets of the form $\Nn_C(f)$ in the following sense.
\begin{defi} Let $G$ be a group and $f\colon G \to \RR$ be a homogeneous quasimorphism. An order on  $G$ is  \emph{sandwiched} 
(or {\em $C$-sandwiched}) by $f$ if there exists a constant $C\in \RR$ such that
\[
\Nn_C(f) \subset G^+\,.
\]
A bi-invariant partial order is sandwiched by $f$, if its order semigroup is sandwiched by $f$.
\end{defi}
This terminology is justified by the following:

\begin{lemma}\label{lem:gp}  If an order on a group $G$ is sandwiched by a quasimorphim  $f\colon G\to\RR$, then
\bqn
G^+\subset\Nn_0(f)\,.
\eqn
\end{lemma}

\begin{proof} We need to show that if $g\in G^+$, then $f(g)\geq0$.
We may assume that $f(g)\neq0$, otherwise we are done.
Since $f$ is homogeneous, for $p\in\NN$ large enough
we have that $|f(g^p)|=|pf(g)|\geq C$.  If $f(g)$ were to be negative,
then $f(g^{-p})=-pf(g)\geq C\geq0$, and hence by hypothesis
$g^{-p}\in G^+$.  Since $G^+$ is a semigroup, it would contain 
both $g^p$ and $g^{-p}$, whence $g^p \in G^+ \cap (G^+)^{-1}=\{e\}$ and hence $f(g) = \frac 1 p f(g^p) = 0$, a contradiction.
\end{proof}
Note that the lemma applies in particular to dominant semigroups.
%If $\Nn_C(f) \subset G^+$ for an order semigroup $G$ then we also say that $f$ is {\em $C$-sandwiched} by the quasimorphism $f$ and refer to $C$ as a {\em sandwiching constant} for $f$.
%
%Note that if $C$ is a sandwiching constant, then any $C'>C$ is also 
%a sandwiching constant.  In particular we can always  take $C>0$. 
The following observation was made in \cite[Lemma 2.10]{BenSimon_Hartnick_JOLT}:
\begin{lemma}\label{lem:sandwich}  Assume that an order semigroup $G^+$ is sandwiched by a homogeneous quasimorphism $f$.
Then 
\bqn
G^{++}=\{g\in G^+\colon\,f(g)>0\}\,.
\eqn
In particular, if a homogeneous quasimorphism sandwiches $G^+$ then it also sandwiches its dominant set and the corresponding Archimedean order.
\end{lemma}

\begin{proof}  The quasimorphism property implies that for all $g,h\in G$ and for all $p\geq1$
\bqn%\label{eq:defect}
pf(g)\geq f(g^ph^{-1})+f(h)-\|df\|_\infty\,.
\eqn
Choose $h\in G$ such that $f(h)\geq\|df\|_\infty$.
If $g\in G^{++}$, there exists $p\geq1$ such that $g^p\succeq h$.  
Hence $g^ph^{-1}\succeq e$ and, by Lemma~\ref{lem:gp}, $f(g^ph^{-1})\geq0$.
Thus $f(g)>0$ and $G^{++}\subseteq\{g\in G^+\colon\,f(g)>0\}$.

Conversely, to see that $\{g\in G^+\colon\,f(g)>0\}\subseteq G^{++}$,
we need to show that if $g\in G^+$ with $f(g)>0$,
then, given $h\in G$, there exists $p\geq1$ such that $g^p\succeq h$ or,
equivalently, $g^ph^{-1}\succeq e$.  Because $G^+$ is sandwiched
by $f$, and hence $\Nn_C(f)\subset \gp$ for some $C\geq0$,
it will be enough to show that $f(g^ph^{-1})\geq C$.
But the quasimorphism property implies that
\bqn
f(g^ph^{-1})\geq pf(g)-f(h) -\|df\|_\infty\,,
\eqn
and, since $f(g)>0$, there exists $p\geq1$ such that 
\bqn
pf(g)-f(h) -\|df\|_\infty\geq C\,.
\eqn
\end{proof}
It is immediate from the definitions that for dominant semigroups, sandwiching by quasimorphisms is functorial. We record the relevant functoriality property for later reference:
\begin{lemma}\label{SandwichFunctorial} Let $f\colon H \to \RR$ be a homogeneous quasimorphism and $H^{++}<H$ be a dominant semigroup which is %$C$-
sandwiched by $f$. Then for every group homomorphism $\rho\colon G \to H$ the pre image $\rho^{-1}(G^{++})$ is sandwiched by $f \circ \rho$.
\end{lemma}

Our next goal is to show that a quasimorphism can be reconstructed, up to a multiplicative constant, from any order it sandwiches. Thus let us fix a bi-invariant partial order $\leq$ on a group $G$ and let us assume that $G^{++} \neq \emptyset$. Let now $g\in G^{++}$, and $h\in G$.  For every $n\geq1$, the set
\bqn
E_n(g,h):=\{p\in\ZZ\colon\, g^p\succeq h^n\}\,.
\eqn
By definition of dominants, the set $E_n(g,h)$ is not empty.  
Moreover it is easy to see that it is bounded below:
in fact,  if $g^{-k}\succeq h^n$ for all $k>0$, then also
$g^k\preceq h^{-n}$ for all $k>0$,
contradicting the fact that $g$ is a Archimedean. 
We may thus define 
\[e_n(g,h):=\min E_n(g,h).\]
Since $g^{k+1}\succeq g^k$ for all $k\in \ZZ$,  then
\bqn
E_n(g,h)=\big[e_n(g,h),\infty\big)\cap\ZZ.
\eqn
Observe that if $g^{p_1}\succeq h^n$ and $g^{p_2}\succeq h^m$,
then
$g^{p_1+p_2}\succeq h^{n+m}$.  Thus
\bqn
e_{n+m}(g,h)\leq e_n(g,h)+e_m(g,h)
\eqn
and hence the limit
\bqn
e(g,h):=\lim_{n\to\infty}\frac{e_n(g,h)}{n}
\eqn
exists. In \cite{Eliashberg_Polterovich} the function $e\colon G^{++} \times G \to \RR$ was referred to as the \emph{relative growth function} of the order $\leq$.
The idea that the function $e$ can be used to reconstruct quasimorphisms from an associated dominant set 
was developed in  \cite{BenSimon_Hartnick_Comm}, 
where in particular the following reconstruction theorem was first established:
\begin{thm}\label{thm:growth function}
Let $G$ be a group and $\gp < G$ an order semigroup.
If $\gp$ is $C$-sandwiched by a homogeneous quasimorphism $f$,
then for every $n\in\NN$,
\bq\label{eq:estimate}
       \frac{-\|df\|_\infty}{nf(g)}
\leq\frac{e_n(g,h)}{n}-\frac{f(h)}{f(g)}
\leq\frac{\|df\|_\infty+C+f(g)}{nf(g)}\,,
\eq
for every $g\in G^{++}$ and $h\in G$.
\end{thm}
We include a proof of Theorem \ref{thm:growth function} for completeness' sake. However,
before we turn to the proof we draw some immediate consequences:
\begin{cor}\label{cor:sandwich} Let $G$ be a group and $\gp < G$ an order semigroup.
\be
\item If $\gp$ is sandwiched by a homogeneous quasimorphism $f$, then 
\bqn
e(g,h)=\frac{f(h)}{f(g)}
\eqn
for all $g\in G^{++}$ and $h\in G$.
\item If $\gp$ is sandwiched by  homogeneous quasimorphisms $f_1, f_2$,
there is $\lambda>0$ such that $f_2=\lambda f_1$.
\ee
\end{cor}
Combining this with Lemma~\ref{SandwichFunctorial} we have proven Theorem~\ref{thm:thm1.5}.

We now turn to the proof of Theorem~\ref{thm:growth function}:
\begin{proof}[Proof of Theorem~\ref{thm:growth function}] 
For all $p,n\in\ZZ$, the quasimorphism inequality reads
\bq\label{eq:qm}
       -\|df\|_\infty-nf(h)+pf(g)
\leq f(g^ph^{-n})
\leq pf(g)-nf(h)+\|df\|_\infty\,.
\eq
Remark that if $g\in G^{++}$, then $f(g)>0$.

The first inequality in \eqref{eq:estimate} follows immediately
from the fact that if $p\geq e_n(g,h)$, then
\bqn
0\leq f(g^ph^{-n})\leq pf(g)-nf(h)+\|df\|_\infty\,.
\eqn

To show the second inequality in \eqref{eq:estimate},
observe that for all $p\in\ZZ$ such that the left hand side
of \eqref{eq:qm} is $\geq C$, that is for all $p\in\ZZ$ such that
\bq\label{eq:p_n}
p\geq\frac{nf(h)+C+\|df\|_\infty}{f(g)}\,,
\eq
we have that $g^ph^{-n}\in\Nn_C(f)\subset\gp$.
Thus $p\geq e_n(g,h)$ and hence
\bqn
\frac{nf(h)+C+\|df\|_\infty}{f(g)}\geq e_n(g,h)-1\,.
\eqn
\end{proof}

\section{Order preserving representations}\label{SecGeneral}

\subsection{The Poincar\'e translation number}\label{SecHuber}
Consider the group $\homeo{\RR}$ of orientation preserving homeomorphisms on the real line. We define a bi-invariant partial order on $\homeo{\RR}$ by saying that $f\succeq g$ if $f(t)\geq g(t)$ for all $t\in\RR$.  Note that if we had considered arbitrary homeomorphisms of $\RR$, the order 
would have been only right invariant. An important subgroup of $\homeo{\RR}$ is given by 
\[\homrp := \{g \in \homeo{\RR}\mid g(x+n) = g(x) + n \text{ for all } n \in \mathbb Z\}, \] which can be identified with the universal cover of the group of orientation preserving homeomorphism of the circle $S^1= \RR/\ZZ$. 
Recalll that the Poincar\'e translation number is the homogeneous quasimorphism $\tau\colon \homrp \to \RR$ given by
\[
\tau(g) = \lim_{n \to \infty}\frac{g^n.x-x}{n},
\]
where $x \in \RR$ is an arbitrary basepoint. The following lemma shows that the translation number sandwiches the restriction of the order $\preceq$ to $\homrp$ (cf. \cite[Prop. 2.16]{BenSimon_Hartnick_JOLT}):
\begin{lemma}\label{lem:homeo}  Let $g\in \homrp$.  Then the  following assertions are equivalent:
\be
\item $g(x)>x$ for all $x\in\RR$;
\item $\tau(g)>0$\,.
\ee
\end{lemma}
\begin{proof}(2)$\Rightarrow$(1):  If (1) fails, then either $g(x)-x<0$ for all $x\in\RR$ or $g(x)-x$ changes sign.
In the first case $g^n(x)<x$ and thus $\tau(g)\leq0$. 
In the second case, by the Intermediate Value Theorem there is 
$x_0\in\RR$ with $g(x_0)=x_0$ and thus $\tau(g)=0$.

\medskip
\noindent
(1)$\Rightarrow$(2):  If $g^n$ denotes the $n$-th iterate of $g$,
then $\tau(g)=\lim_{n\to\infty}\frac{g^n(x)-x}{n}$ for every $x\in\RR$.
In view of (1), the sum
\bqn
g^n(x)-x=\sum_{i=0}^{n-1}(g(g^i(x))-g^i(x))
\eqn
consists of positive terms.  Thus if $\tau(g)=0$, there is a subsequence $i_k\to+\infty$
with 
\bqn
\lim_{k\to\infty}(g(g^{i_k}(x))-g^{i_k}(x))=0\,.
\eqn
Thus, since $g$ commutes with integer translations,
if $\{\,\cdot\,\}$ denotes the fractional part of a real number, then
\bqn
\lim_{k\to\infty}(g(\{g^{i_k}(x)\})-\{g^{i_k}(x)\})=0\,.
\eqn
If now $y\in[0,1]$ is an accumulation point of the sequence $\{g^{i_k}(x)\}_{k\geq1}$,
then $g(y)=y$, which contradicts (1).
\end{proof}
Taking into account Lemma~\ref{lem:sandwich}, and setting $\check G = \homrp$ we deduce that 
\bqn
\ba
\check G^{++}
=\{g\in \check G\colon\,\tau(g)>0\}
=\{g\in \check G\colon\, g(x)>x\text{ for all }x\in\RR\}\\
\subsetneq \check G^+
\subsetneq\{g\in \check G\colon\, \tau(g)\geq0\},
\ea
\eqn
which in particular shows that $\tau$ sandwiches $\preceq$ as claimed. 

Moreover, the Archimedean order associated with $\check G^{++}$ is 
the order $\leq_0$ on $\homrp$ introduced in \S~\ref{sec:1.1}.

\subsection{The order on fundamental groups of surfaces}\label{SubsecOrderSurface}

The standard action of ${\rm PSL}_2(\RR)$ on $\mathbb{RP}^1\cong S^1$ gives rise to an action of the universal cover $\widetilde{\rm PSL}_2(\RR)$ on $\RR$ commuting with integral translations. 
Hence we may consider $\widetilde{\rm PSL}_2(\RR)$ as a subgroup of $\homrp$. We denote by $\tau\colon \widetilde{\rm PSL}_2(\RR) \to \RR$ the restriction of the translation number to $\widetilde{\rm PSL}_2(\RR)$  and write $\leq_0$ for the Archimedean order on $\widetilde{\rm PSL}_2(\RR)$ induced from the Archimedean order on $\homrp$. 

Given a hyperbolization $\rho_h\colon \pi_1(\Sigma) \to {\rm PSL}_2(\RR)$ we consider the canonical lift $ \widetilde{\rho}_h\colon \Lambda_\Sigma \to \widetilde{\rm PSL}_2(\RR)$.  Composition with the translation number defines a homogeneous quasimorphism on $\Lambda_\Sigma$: 
\[ \tau \circ \widetilde{\rho}_h\colon \Lambda_\Sigma \to \RR\]

\begin{lemma}\label{lem:fundamental_class}
The quasimorphism 
\[f_\Sigma:= \tau \circ \widetilde{\rho}_h\colon \Lambda_\Sigma \to \RR\]
is independent of the chosen hyperbolization $\rho_h$. 
\end{lemma}
\begin{proof}
As alluded to above, the boundary of the translation number $[d\tau]$ defines a (continuous) bounded cohomology class $[d\tau] \in H^2_{cb}( \widetilde{\rm PSL}_2(\RR); \RR)$ 
in the second bounded cohomology of $\widetilde{\rm PSL}_2(\RR)$. 
Let $\kappa^b_{\rm PSL_2(\RR)}$ denote the bounded K\"ahler class of ${\rm PSL}_2(\RR)$, and let $p^*\kappa^b_{\rm PSL_2(\RR)}$ be the pullback under the projection $p\colon \widetilde{\rm PSL}_2(\RR) \to {\rm PSL}_2(\RR)$, then 
\[
p^*\kappa^b_{\rm PSL_2(\RR)} = [d\tau] \in H^2_{cb}( \widetilde{\rm PSL}_2(\RR); \RR).
\]

It was proved in \cite[\S~8.2]{Burger_Iozzi_Wienhard_tol} that for a hyperbolization $\rho_h\colon \pi_1(\Sigma) \to {\rm PSL}_2(\RR)$, the bounded cohomology class $\rho_h^*\kappa^b_{\rm PSL_2(\RR)} \in H^2_{b}( \pi_1(\Sigma); \RR)$ is independent of the choice of $\rho_h$. 
Thus, given any two hyperbolic structures $h_1$ and $h_2$, 
\bqn
d(\tau\circ\widetilde\rho_{h_1})=d(\tau\circ\widetilde\rho_{h_2})
\eqn
and hence $\tau\circ\widetilde\rho_{h_1}$ and $\tau\circ\widetilde\rho_{h_2}$ must coincide on $\Lambda_\Sigma$.
\end{proof}

\begin{rem}
 A geometric interpretation of the quasimorphism $f_\Sigma$ as a generalized winding number was given by Huber in \cite{Huber}, following earlier work of Chillingworth \cite{Chillingworth_1, Chillingworth_2}.
\end{rem}

In view of Lemma \ref{dominantsFunctorial} and Lemma \ref{lem:homeo} we have now the following consequence.

\begin{prop}\label{dominantPSL} If $\rho_h\colon\Gamma \to{\rm PSL}_2(\RR)$ is a hyperbolization as above then
\begin{eqnarray*}
\Lambda^{++}_\Sigma &:=& \{\gamma\in \Lambda\mid \widetilde{\rho}_h(\gamma).x>x \text{ for all }x\in \RR\}\\
&=& \{\gamma \in \Lambda\mid f_\Sigma(\gamma)>0\}
\end{eqnarray*}
is a dominant semigroup, which is independent of the hyperbolization used to define it.
\end{prop}
We also record the following easy fact for later reference:
\begin{prop}\label{PropqSandwich} For every $q>0$ the set
\begin{eqnarray*}
\Lambda^{++}_{q,\Sigma} &:=& \{\gamma\in \Lambda\mid \widetilde{\rho}_h(\gamma).x>x+q \text{ for all }x\in \RR\}\\
&=& \{\gamma \in \Lambda\mid f_\Sigma(\gamma)>q\}.
\end{eqnarray*}
is a dominant semigroup, which is sandwiched by $f_\Sigma$. In particular, $\Lambda^{++}_\Sigma=\Lambda^{++}_{\Sigma,0}$ is sandwiched by $f_\Sigma$.
\end{prop}
\begin{proof} The first equality follows from Lemma \ref{lem:homeo}. It then follows from Example \ref{ExQm} that $\{e\} \cup \Lambda^{++}_{q,\Sigma}$ is an order semigroup sandwiched by $f_\Sigma$. The fact that every element in $\Lambda^{++}_{q,\Sigma}$ is dominant follows from Lemma \ref{lem:sandwich}.
\end{proof}

Note that Theorem \ref{thm:InvariantOrder} of the introduction is an immediate consequence of Proposition \ref{dominantPSL}  and Proposition \ref{PropqSandwich}. At this point we have thus established that hyperbolizations are order preserving representations. In order to show that conversely any order preserving representation is a hyperbolization, we will consider the more general setting of order preserving representations into Lie groups of Hermitian type 

\subsection{Order preserving representations}\label{SecMeta}
\begin{proof}[Proof of Theorem~\ref{thm:thm1.1}]
The implications (1)$\Rightarrow$(2) and (3)$\Rightarrow$(2) are the content of Propositions~\ref{dominantPSL}
and \ref{PropqSandwich}.  Assume now that (2) holds.  Then it follows from Theorem~\ref{thm:thm1.5}
that there is $\lambda>0$ with
\bqn
\tau\circ\widetilde\rho=\lambda f_\Sigma\,.
\eqn
Now since $d\tau$ and $df_\Sigma$ take values in $\{-1,0,1\}$, 
this implies that $\lambda=1$. As a result $\rho\colon\pi_1(\Sigma)\to\PSL(2,\RR)$
is a maximal representation and the result follows from \cite[Theorem~3]{Burger_Iozzi_Wienhard_tol}.
\end{proof}

\begin{proof}[Proof of Corollary~\ref{cor:cor1.2}]
Let $i\colon\pi_1(\Sigma_1)\to\pi_1(\Sigma_2)$ be an isomorphism,
with $i:\Lambda_{\Sigma_1}\to\Lambda_{\Sigma_2}$ strictly order preserving.
Let $g:\Sigma_1^\circ\to\Sigma_2^\circ$ be a continuous map implementing $i$,
and $h$ a hyperbolic structure on $\Sigma_2^\circ$ with holonomy
$\rho_h^{(2)}\colon\pi_1(\Sigma_2)\to\PSL(2,\RR)$ and developing map
\bqn
f_h^{(2)}\colon\Sigma_2^\circ\to\rho_h^{(2)}(\pi_1(\Sigma_2))\backslash \DD,
\eqn
which is a diffeomorphism.  
Then the lift $\widetilde{\rho_h^{(2)}}\circ i\colon\Lambda_\Sigma\to\widetilde{PSL}(2,\RR)$
is strictly order preserving and hence there is a complete hyperbolic structure $h'$ on $\Sigma_1^\circ$
with $\rho_{h'}^{(1)}=\rho_h^{(2)}\circ i$.  If now
\bqn
f_{h'}^{(1)}\colon\Sigma_1^\circ\to\rho_{h'}^{(1)}(\pi_1(\Sigma_1))\backslash\DD=\rho_{h}^{(2)}(\pi_1(\Sigma_2))\backslash\DD
\eqn
is the developing diffeomorphism, then $f_{h'}^{(1)}$ is homotopic to $f_h^{(2)}\circ g$ 
and hence $g$ is homotopic to a diffeomorphism.
\end{proof}

Now we turn to homomorphisms with values in a connected simple group $G$ of Hermitian type
and with finite center.  As before we let $\check G$ denote the connected $\ZZ$-covering of $G$
and $f_{\check G}\colon\check G\to \RR$ the continuous homogeneous quasimorphism such that
$[df_{\check G}]$, seen as a class on $G$, represents the bounded K\"ahler class $\kgb$.

We will prove here the following result of which Theorem~\ref{thm:thm1.6} will be a corollary.

\begin{thm}\label{thm:thm3.6}  Let $\check G^+$ be an Archimedean order on $\check G$ sandwiched by $f_{\check G}$.
Then there is $q_0:=q_0(G,\check G,\check G^+)$, such that for a representation $\rho\colon\pi_1(\Sigma)\to G$
with lift $\widetilde\rho\colon\Lambda_\Sigma\to\check G$ the following are equivalent:
\be
\item The lift $\widetilde\rho$ is strictly order preserving for $\leq_{q_0,\Sigma}$.
\item There is $q\in\NN$ such that $\widetilde\rho$ is strictly order preserving for $\leq_{q,\Sigma}$.
\item There exists $\lambda>0$ with 
\bqn
f_{\check G}\circ\widetilde\rho |_{\Lambda_\Sigma}=\lambda f_\Sigma\,.
\eqn
\ee
\end{thm}

Theorem~\ref{thm:thm1.6} is then a consequence of \cite{BenSimon_Hartnick_Comm} 
and the following characterization of weakly maximal representations \cite{BBHIW1}:

\begin{prop}\label{prop:prop3.7}  A representation $\rho\colon\pi_1(\Sigma)\to G$ is weakly maximal 
if and only if its lift
\bqn
\widetilde\rho\colon\Lambda_{\Sigma_1}\to\check G
\eqn
satisfies
\bqn
f_{\check G}\circ\widetilde\rho |_{\Lambda_\Sigma}=\lambda f_\Sigma
\eqn
for some $\lambda \in \RR$.  

Moreover in this case $\lambda=|\chi(\Sigma)|^{-1}\T(\rho)$, 
where $\T(\rho)$ is the Toledo invariant of $\rho$.
\end{prop}

\begin{proof}[Proof of Theorem~\ref{thm:thm3.6}]  We define $q_0\in\NN$ as follows:
let $C\in\NN^\ast$ be such that $\Nn_C(f_{\check G})\subset\check G^+$
and let $\ell_G\geq1$ be an integer such that 
\bqn
\ell_G\T(\rho)\in|\chi(\Sigma)|\ZZ
\eqn
for every weakly maximal representation $\rho$ (see \cite[Theorem~1.3]{BBHIW1}); set $q_0:=C\ell_G$.

For this choice of constant we show the implications (2)$\Rightarrow$(3)$\Rightarrow$(1);
since (1)$\Rightarrow$(2) holds trivially, this will finish the proof.

\noindent
(2)$\Rightarrow$(3)  Since $\widetilde\rho$ is strictly order preserving, the orders 
are Archimedean and sandwiched respectively by $f_\Sigma$ and $f_{\check G}$.
Then Theorem~\ref{thm:thm1.5} implies that 
\bqn
f_{\check G}\circ\widetilde\rho|_{\Lambda_\Sigma}=\lambda f_\Sigma
\eqn
for some $\lambda>0$.  

\noindent
(3)$\Rightarrow$(1)  By \cite[Proposition~3.2]{BBHIW1} we have that $\lambda=|\chi(\Sigma)|^{-1}\T(\rho)$ and hence
\bqn
\lambda q_0=\lambda C\ell_G=C\lambda \ell_G\geq G\,,
\eqn
since $\lambda\ell_G=\frac{\T(\rho)\ell_G}{|\chi(\Sigma)|}\in\NN^\ast$.
Then for every $\gamma\geq_{\Sigma,q_0}e$ we obtain
\bqn
f_{\check G}(\widetilde\rho(\gamma))=\lambda f_\Sigma(\gamma)>\lambda q_0\geq C
\eqn
and hence $\widetilde\rho(\gamma)\in\check G^+$, $\widetilde\rho(\gamma)\neq e$.
Thus $\widetilde\rho$ is strictly order preserving for $\leq_{q_0,\Sigma}$.
\end{proof}

\begin{rem}\label{PerturbationStable} Theorem \ref{thm:thm3.6} is actually stable under perturbations of orders in the following sense. Assume that $\check G$ is an Archimedean order on $\check G^+$ sandwiched by $f_{\check G}$, and let
$\check G^+_g$ be a perturbation of  $\check G^+$ in the sense of Lemma \ref{Perturbation}. We claim that $\check G^+_g$ is sandwiched by $f_{\check G}$ as well. 

Indeed, we can choose  $C_0>0$ such that $\Nn_{C_0}(f_{\check G}) \subset \check G^{++}$. Now denote by $\|df\|_\infty$ the defect of $f$ and set $C := C_0 + f_{\check G}(g)+\|df\|_\infty$. If we assume that for some $\gamma \in G$ we have $f_{\check G}(\gamma) \geq C$, then for every $h \in \check G$ we have
\begin{eqnarray*}
f_{\check G}(g^{-1}h^{-1}\gamma h) &\geq& f_{\check G}(g^{-1}) + f_{\check G}(h^{-1}\gamma h) -\|df\|_\infty\\ &\geq& - f_{\check G}(g) + C - \|df\|_\infty = C_0.
\end{eqnarray*}
This shows that $g^{-1}h^{-1}\gamma h \in \check G^{++}$ and thus $\gamma \in h(g \check G^+)h^{-1}$. Since $h$ was chosen arbitrarily we deduce that $\gamma \in \check G^+_g$, which proves our claim.

As a consequence, Theorem \ref{thm:thm1.6} remains valid if we replace the continuous order under consideration by any of its perturbations.
\end{rem}

\section{The Shilov boundary and the causal ordering}\label{sec:causal_order} 
Throughout this section $G$ is a simple adjoint connected Lie group
of Hermitian type and we assume that the associated bounded symmetric
domain $\Dd$ is of tube type.  Recall that the Shilov boundary $\cs\cong G/Q$ is a
homogeneous space for $G$, where $Q$ is an appropriate maximal
parabolic subgroup. It was established by Kaneyuki \cite{Kaneyuki} 
that $\cs$ carries a $G$-invariant causal structure, unique up to inversion. 
The lift of this causal structure to the universal covering $\check R$ of $\cs$
defines an order on $\check R$ invariant under the action of $\check G$. In this section we will use this order on $\check R$ on the one hand 
to define an order on $\check G$ and on the other to define explicitly an integral valued
Borel quasimorphism on $\check G$ whose homogenization  is essentially the
quasimorphism $f_{\check G}$. We then use this construction to show that the order on $\check G$ is sandwiched
by the quasimorphism $f_{\check G}$. We also determine explicitly the set of dominant elements.
A similar construction of quasimorphisms starting with a causal structure 
has been pointed out by Calegari in \cite[5.2.4]{Calegari_scl} and in greater generality
by Ben Simon and Hartnick in \cite{BenSimon_Hartnick_quasitotal}.

\subsection{The Kaneyuki causal structure}\label{Sec:Kaneyuki}
We need to recall some basic properties of Kaneyuki's construction. 
Observe that any maximal compact subgroup $K$ of $G$ acts transitively on $\cs$. 
We denote by $M$ the stabilizer of the basepoint $o=eQ$, so that $\cs \cong K/M$. 
Consider now the holonomy representation of $\rho\colon M \to GL(V)$ on $V := T_{eQ}\cs$. 
According to \cite{Kaneyuki}, there exists an inner product $\langle \cdot, \cdot \rangle$ and an open cone $\Omega \subset V$ such that
\[\rho(M) = \{g \in O(V, \langle \cdot, \cdot \rangle)\,|\, g\Omega =\Omega\}.\]
The cone $\Omega$ is symmetric with respect to $\langle \cdot, \cdot \rangle$ in the sense of \cite[p. 4]{Faraut_Koranyi}. 
In particular $\Omega$ is self-dual, i.e.
\bq\label{SelfDualCone}
\Omega =\{w \in V\,|\, \langle v, w \rangle > 0 \text{ for all } v \in \overline{\Omega}\setminus \{0\}\}\,.
\eq
Furthermore, we deduce from \cite[Prop. I.1.9]{Faraut_Koranyi} 
that $\rho(M)$ fixes a unit vector $e$ in $\Omega$. 
As a consequence,  there exists a constant $k$ such that for all $w\in\overline\Omega$,
\bq\label{eq:uniform bound}
\langle e,w\rangle\geq k\|w\|\,.
\eq
Since $\langle \cdot, \cdot \rangle$, $e$ and $\overline{\Omega}$ are $M$-invariant 
they give rise respectively to a $K$-invariant Riemannian metric $g$, a $K$-invariant vector field $v$, 
and a $K$-invariant causal structure $\Cc$, all uniquely determined by their
definition at the basepoint
\bq\label{eq:spread}
g_o = \langle \cdot, \cdot \rangle, \quad v_o = e, \quad C_o = \overline{\Omega}\,.
\eq
%Observe that, in view of \eqref{SelfDualCone} for every $p \in \cs$,
%\bqn
%C_p = \{v \in T_p\cs\,|\, g_p(v, w) \geq 0 \text{ for all } w \in C_p\}\,.
%\eqn
For future reference we record here the following observation that follows
immediately from \eqref{eq:uniform bound} and the fact that the objects in \eqref{eq:spread} 
are $K$-invariant.
\begin{lemma}\label{lem:acute}  The cones $C_p$ are uniformly acute with respect to $g_p$, that is
there exists $k>0$ such that 
\bqn
g_p(w,v_p)\geq k\|w\|\,,
\eqn
for all $w\in C_p$ and for all $p\in\cs$.
\end{lemma}
According to \cite{Kaneyuki}, the $K$-invariant causal structure defined above is $G$-invariant.
% cone fields $\pm C_p$ are precisely the two $G$-invariant causal structures on $\cs$. 
In the sequel we will also need to consider various coverings of $(\cs, \Cc)$. 
To this end we recall \cite{Clerc_Koufany} that $\pi_1(\cs)\cong\ZZ$, 
from which it is easily deduced that also $\pi_1(G)$ has rank one. 
However, in general the map $\pi_1(G)\to\pi_1(\cs)$ is not surjective or,
in other words, $Q$ may not be connected. 
To deal with this problem we introduce the finite covering $S':=G/Q^\circ$ of
$\cs$ and observe that the evaluation map 
\bqn 
\ba
G&\to S'\\
g&\mapsto gQ^\circ 
\ea 
\eqn induces an isomorphism 
\bq\label{eq:pi}
\pi_1(G)/\pi_1(G)_\mathrm{tor}\cong\pi_1(S')\,.
\eq
It follows that, if $\check G$ denotes the
connected covering of $G$ associated to $ \pi_1(G)_\mathrm{tor}$ as before,
then $\check G$ acts faithfully on the universal covering $\check R$ of $S'$, 
covering the $G$-action on $S'$.  
Moreover it follows from \eqref{eq:pi} that the action of the center $Z(\check G)$ of $\check G$
coincides with the $\pi_1(S')$-action.
Now the causal structure $\Cc$ on $\cs$ lifts to causal structures on $S'$ and $\check R$, 
invariant under the corresponding automorphism groups. 
We will abuse notation and denote all these causal structures by $\Cc$. 
With this abuse of notation understood, 
Lemma \ref{lem:acute} holds also for $S'$ and $\check R$.
\subsection{The causal order on $\check R$}\label{Sec:KanyukiOrder}
We recall the definition of a causal curve:
\begin{defi}\label{defi:causal curve}  Let $M$ be any covering of $\cs$.  A curve $\gamma\colon[a,b]\to M$ is {\em causal} if
it is piecewise $C^1$, with existing left and right tangents $\dot\gamma(t_+)$ and $\dot\gamma(t_-)$ 
at all points $t\in [a,b]$, and such that $\dot\gamma(t_\pm)\in C_{\gamma(t)}$ for all $t\in[a,b]$.
\end{defi}

Using causal curves, we now define a relation on $\check R$ as follows:
\begin{defi}\label{defi:causal order}  Let $x,y\in\check R$.  We say that 
$x\leq y$ if there is a causal curve from $x$ to $y$ and write $x < y$ if $x\leq y$ and $x \neq y$.
\end{defi}  
The relation $\leq$ is obviously transitive and $\check G$-invariant (since the underlying causal structure is), 
and we will see in Lemma~\ref{lem:orderinR} that it gives a $\check G$-invariant partial order on $\check R$.

We now define a $K$-invariant $1$-form $\alpha$ on $S'$ by 
\bq\label{eq:form}
\alpha_p(x):=g_p(x,v_p)\,,
\eq
for $x\in T_p{S'}$. (Here and in the sequel we abuse notation to denote the lifts of $v$ and $g$ from $\cs$ to $S'$ by the same letters.) 
Since $S'$ is symmetric and $\alpha$ is $K$-invariant we have, by Cartan's lemma,
\bq\label{eq:closed}
d\alpha=0\,.
\eq
Furthermore, let $p\colon\check R\to S'$ be the universal covering map,
$\check\alpha:=p^\ast(\alpha)$, and fix a smooth function $\zeta\colon\check R\to\RR$
such that $d\zeta=\check\alpha$.  If $\gamma\colon[0,1]\to\check R$ is a causal curve joining
$x$ to $y$, then it follows from Lemma~\ref{lem:acute} that there exists a constant $k>0$ such that
\bq\label{eq:integral}
\zeta(y)-\zeta(x)=\int_\gamma\check\alpha\geq k\cdot\mathrm{Length}(\gamma)\geq k\cdot d_{\check R}(x,y)\,,
\eq
where $d_{\check R}$ refers to the Riemannian metric on $\check R$ corresponding to the one on $S'$.
This gives immediately the following:

\begin{lemma}\label{lem:orderinR}  The relation $\leq$ is a partial order on $\check R$.
\end{lemma}

\begin{proof} The relation is transitive since the concatenation of causal curves is causal.
If now $x\leq y$ and $y\leq x$, it follows from \eqref{eq:integral} that 
$\zeta(y)-\zeta(x)\geq0$ and $\zeta(x)-\zeta(y)\geq0$, hence $d_{\check R}(x,y)=0$.
\end{proof}

We record for future purposes that, integrating over a geodesic path from $x$ to $y$ and using that $v$ is unit length, 
we have also the inequality
\bq\label{eq:bound}
\zeta(y)-\zeta(x)\leq d_{\check R}(x,y)\,.
\eq

We now exploit the fact that $\leq$ is an order on $\check R$ but not on $S'$,
where causal curves can be closed.  However we first need to introduce some
objects that stem from an alternative description of the vector field $v$.  
To this purpose observe that the action of $Z(K)$ on $S'$ gives a locally free $S^1$-action; 
the unit length vector field generating this action is precisely $v_p$, see \cite[Proposition~I.1.9]{Faraut_Koranyi}.
Since $v_p\in C_p^\circ\subset C_p$, the integral curve 
\bq\label{eq:integralcurve} 
\gamma_p\colon[0,L]\to S' 
\eq
going through the point $p\in S'$ is a closed causal curve of length $L$ and 
hence $\gamma_p$ generates a subgroup of finite index in $\pi_1(S')$.  
Let $\hhs$ be the finite covering of $S'$ corresponding to this subgroup.  
For every $p\in\hhs$, the unique lift of $\gamma_p$ to a closed curve in $\hhs$ through
$p$ is now a generator $Z$ of $\pi_1(\hhs)$ that is causal and has also length $L$.

\begin{lemma}\label{lem:maxlengthofcausalcurves}  There exists a constant $D>0$ 
depending only on the Riemannian metric $g$ on $\hhs$ such that 
any two points in $\hhs$ can be joined by a causal path of length at most $D$.
\end{lemma}

\begin{proof} First we show that there exists $\eta>0$ such that 
every $q\in\hhs$ can be joined to the points in the open $\eta$-ball in $\hhs$
   with center in $q$  by a causal path of length at
  most $L+1$, where $L$ is defined in \eqref{eq:integralcurve}.  
  Note that the causal curve will not necessarily be contained in the ball.
  
  Let $K''$ be an appropriate covering of $K$ 
  that acts effectively and transitively on $\hhs$, 
  and let $M''$ be the stabilizer of $q$.  
  For a suitable generator $Y$ of the Lie algebra of $Z(K'')$, 
  we have that 
  \bqn
  \gamma_q(t)=\operatorname{Exp} tY=\exp tv_q\,,
  \eqn
  where $\exp$ is the Riemannian exponential map.  
  As a result, $\exp Lv_q=q$.
  
  Let now $0<\epsilon<1$ be such that the ball $B_\epsilon(Lv_q)$ with radius $\epsilon$
  in the tangent space at $q\in\hhs$ is contained in $C_q$ and
  let $\eta>0$ such that $B_\eta(q)\subset\exp(B_\epsilon(Lv_q))$.
  Given now any $v\in C_q$, we observe that the curve $t\mapsto\exp tv$, for $t\geq0$,
  is causal.  Indeed, the causal structure is left invariant by the parallel transport, 
  since it is realized by left multiplication of appropriate elements of $K''$.
  Thus for $v\in C_q$ of length $1$ such that  $\ell_vv\in B_\epsilon(Lv_q)$, 
  the assignment $t\mapsto\exp tv$ from $[0,\ell_v]$ to $\hhs$
  gives the causal curve joining $q$ to $\exp \ell_vv$.

  Let then $d=\diam(\hhs)$.  Given two points $p,q\in\hhs$, choose
  a distance minimizing geodesic $c\colon[0,d(p,q)]\to\hhs$ parametrized
  by arc length. Then it follows from the above claim that for every
  $n\in\NN$ with $n\frac{\eta}{2}\leq d(p,q)$, $c((n-1)\eta/2)$ can be
  joined to $c(n\eta/2)$ by a causal path of length at most $L+1$.
  Thus $p=c(0)$ can be joined to $q=c(d(p,q))$ by a causal path of
  length at most
  $D:=\left(\left[\frac{2d}{\eta}\right]+1\right)(L+1)$.
\end{proof}

\subsection{The generalized translation number}\label{Sec:GTransNum}
With the aid of the invariant order on $\check R$, 
we now proceed to the construction of the quasimorphism on $\check G$.
We start by observing that since $Z\in\pi_1(\hhs)$ can be represented 
by the closed causal curve $\gamma_p$ for every $p\in \hhs$,
then
\bq\label{eq:Z(a)geqa}
Z(x)\geq x
\eq for every $x\in\check R$.

Given $x,y\in\check R$, let us define
\bqn%\label{eq:ixy}
I(x,y):=\{n\in\ZZ\colon\,Z^ny\geq x\}\,.
\eqn

It follows immediately from \eqref{eq:Z(a)geqa} that 
if $m\in I(x,y)$, then $Z^{m+1}y\geq Zx\geq x$, which shows that $m+1\in I(x,y)$.
In order to study further properties of the set $I$, we will need the function $\zeta$ defined above, 
in particular the property  that, for all $x\in\check R$, and all $n\in\ZZ$, 
\bq\label{eq:q-m}
\zeta(Z^nx)=n\,L+\zeta(x)
\eq
since $L>0$ is the integral of $\alpha$ over $\gamma_p$.

\begin{lemma}\label{lem:height}  With the above notation
\be
\item $\iota(x,y):=\min I(x,y)$ is well defined and $I(x,y)=[\iota(x,y),\infty)\cap\ZZ$.
\item $\iota$ is invariant under the diagonal $\check G$-action on $\check R\times\check R$.
\item $|L\iota(x,y)-(\zeta(x)-\zeta(y))|\leq D$ for all $x,y\in\check R$, where $D$ is given by 
Lemma~\ref{lem:maxlengthofcausalcurves}.
\ee
\end{lemma}
\begin{proof}  (1) If $n\in I(x,y)$, then, by \eqref{eq:integral},
\bqn
\zeta(Z^ny)-\zeta(x)\geq0\,,
\eqn
and thus, by using \eqref{eq:q-m}
\bqn
n\,L-(\zeta(x)-\zeta(y))\geq0\,.
\eqn
Thus $n \geq\frac{1}{L}(\zeta(x)-\zeta(y))$, 
and since $n\in I(x,y)$ is arbitrary, then 
\bq\label{eq:ineqh}
L\iota(x,y)-(\zeta(x)-\zeta(y))\geq0\,.
\eq

\medskip
\noindent
(2) This follows immediately from the fact 
that the $\check G$-action commutes with $Z$ and preserves the order.

\medskip
\noindent
(3) Let  $c\colon[0,D]\to\hhs$ be a causal curve of length at most $D$
joining $p(x)$ to $p(y)$, 
where $p\colon\check R\to\hhs$ is the canonical projection (see Lemma~\ref{lem:maxlengthofcausalcurves}).
Let $\tilde c\colon[0,D]\to\check R$ be the unique continuous lift of $c$ with $c(0)=x$.
Then $\tilde c(D)=Z^ny$ for some $n\in\ZZ$.  Applying \eqref{eq:bound}, we get 
\bqn
\zeta(Z^ny)-\zeta(x)\leq d_{\check R}(Z^ny,x)\leq D\,,
\eqn
which, taking into account \eqref{eq:q-m}, implies that 
\bqn
n\, L+\zeta(y)-\zeta(x)\leq D\,,
\eqn
that is 
\bqn
L\,\iota(x,y)-(\zeta(x)-\zeta(y))\leq D\,.
\eqn
This inequality and \eqref{eq:ineqh} conclude the proof.
\end{proof}
The function $\iota$ is an example of an abstract height function of a causal covering 
in the sense of \cite{BenSimon_Hartnick_quasitotal}, where a general theory of such functions is developed. 
It follows from this general theory that for any $x\in\check R$ the function $R_x\colon\check G\to\ZZ$ given by 
\bqn
R_x(g):=\iota(gx,x)
\eqn
is a quasimorphism and that all these quasimorphisms are mutually at bounded distance.  In the present case it is actually easy to derive these properties directly:
\begin{lemma}\label{lem:propertiesh}  For all $x,y\in\check R$, 
all $g,h\in\check G$ and all $n\in\ZZ$, we have:
\be
\item $R_x(Z^n)=n$;
\item $0\leq R_x(g)+R_x(g^{-1})\leq 2D/L$;
\item $|R_x(gh)-R_x(g)-R_x(h)|\leq 3D/L$;
\item $|R_x(g)-R_y(g)|\leq 4D/L$.
\ee
\end{lemma}

\begin{proof} (1) follows from the fact that $Z(x)\geq x$ and $Z^{-1}(x)\leq x$ 
for all $x\in\check R$.

\medskip
\noindent
(2)  
Using Lemma~\ref{lem:height}(3), we get
\bqn
L|\iota(x,g^{-1}x)+\iota(g^{-1}x,x)|\leq2D\,,
\eqn
which implies the right inequality since $\iota(x,g^{-1}x)=\iota(gx,x)$.

To see the left inequality observe that 
if $Z^ny\geq x$ and $Z^mx\geq y$, then $Z^{n+m}y\geq Z^mx\geq y$, and 
\eqref{eq:Z(a)geqa} implies that $n+m\geq0$.

\medskip
\noindent
(3)  Using repeatedly Lemma~\ref{lem:height}(2) and (3) we obtain:
\bqn
\ba
   &L|R_x(gh)-R_x(g)-R_x(h)|\\
=&|L\iota(hx,g^{-1}x)-L\iota(x,g^{-1}x)-L\iota(hx,x)|\\
=&\big|[L\iota(hx,g^{-1}x)-(\zeta(hx)-\zeta(g^{-1}x))]\\
	&\qquad+[L\iota(x,g^{-1}x)-(\zeta(x)-\zeta(g^{-1}x))]\\
	&\qquad+[L\iota(hx,x)-(\zeta(hx)-\zeta(x))]\big|\leq 3D\,.
\ea
\eqn

\medskip
\noindent
(4)  Again from Lemma~\ref{lem:height}(2) and (3)
\bqn
\ba
  &L|R_x(g)-R_y(g)|\\
=&L|R_x(g)-\iota(gx,gy)+\iota(x,y)-R_y(g)|\\
=&\big|[L\iota(gx,x)-(\zeta(gx)-\zeta(x))]\\
&\qquad-[L\iota(gx,gy)-(\zeta(gx)-\zeta(gy))]\\
&\qquad+[L\iota(x,y)-(\zeta(x)-\zeta(y))]\\
&\qquad-[L\iota(gy,y)-(\zeta(gy)-\zeta(y))]\big|\leq4D\,.
\ea
\eqn
\end{proof}

We consider now the homogenization
\bqn
\psi(g):=\lim_{n\to\infty}\frac{R_x(g^n)}{n}
\eqn
of the quasimorphism $R_x$. Notice that, because of Lemma~\ref{lem:propertiesh}(4),
$\psi(g)$ does not depend on $x\in\check R$.

\begin{prop}\label{prop:sandwich}  The map $\psi\colon\check G\to\RR$ is a continuous quasimorphism
satisfying the following properties:
\be
\item $\psi(Z^n)=n$, for $n\in\ZZ$;
\item $\|\psi-R_x\|_\infty\leq 3D/L$;
\item $\psi$ sandwiches $G^+$, in fact $\{g\in G\colon\,\psi(g)\geq 5D/L\}\subset G^+$.
\ee
\end{prop}

\begin{proof}  Since $\psi$ is the homogenization of $R_x$,
then (1) and (2) follow respectively from Lemma~\ref{lem:propertiesh}(1) and (3).
The fact that $\psi$ is continuous follows from the fact
that it is a homogeneous Borel quasimorphism \cite[Lemma~7.4]{Burger_Iozzi_Wienhard_tol}.  

If now $\psi(g)\geq5D/L$, then it follows from (2) that $R_x(g)\geq2D/L$.  
Lemma~\ref{lem:propertiesh}(2) then implies that $\iota(x,gx)=R_x(g^{-1})\leq0$ and 
hence $gx\geq x$ for all $x\in\check R$. The assertion follows now from Lemma~\ref{lem:gp}.
\end{proof}

\begin{cor}\label{QMUnique} There exists $\lambda \in \RR^\times$ such that $\psi = \lambda \cdot f_{\check G}$.
\end{cor}
\begin{proof}
Since the space of continuous homogeneous quasimorphisms on $\check G$ is one-dimensional (see \cite{Burger_Iozzi_Wienhard_tol}) there exists $\lambda \in \RR$ such that $\psi = \lambda \cdot f_{\check G}$. Since $\psi(Z) \neq 0$ we have $\lambda \neq 0$.
\end{proof}

\subsection{The dominants and the Archimedean order associated with the causal order}\label{Sec:dominants}

We can now identify the set of dominants of the causal order $\preceq$. The result extends to a general Lie group of Hermitian type
the statement in Lemma~\ref{lem:homeo} for $\homrp$ and the Poincar\'e translation 
quasimorphism $\tau$.  The proof is also very similar to that of Lemma~\ref{lem:homeo},
where here the function $\zeta$ plays an important role.
\begin{prop}\label{lem:dominant}  
$\check G^{++}=\{g\in\check G\colon\,g(x)>x\text{ for all }x\in\check R\}$.
\end{prop}

\begin{proof} Observe first of all that from the definition of $\psi$ and 
Lemma~\ref{lem:height}(3), it follows that 
\bq\label{eq:qmlikeP}
\psi(g)=\lim_{n\to\infty}\left(\frac{\zeta(g^nx)-\zeta(x)}{L\,n}\right)\,.
\eq

From Proposition~\ref{prop:sandwich} and Lemma~\ref{lem:sandwich} 
it follows that $\check G^{++}=\{g\in\check G^+\colon\,\psi(g)>0\}=\{g\in G\colon\,g(x)\geq x\text{ for all }x\in\check R\text{ and }\psi(g)>0\}$.
Hence we need to show the equivalence for an element $g\in \check G$ between
\be
\item $g(x)\geq x$ for all $x\in\check R$ and $\psi(g)>0$, and
\item $g(x)>x$ for all $x\in\check R$.
\ee

To see that (1)$\Rightarrow$(2) it is immediate to verify that if (1) holds and (2) fails, 
then there exists a fixed point $x_0\in\check R$ and hence, because of \eqref{eq:qmlikeP},
$\psi(g)=0$, a contradiction.

To see that (2)$\Rightarrow$(1) we show that if $g(x)\geq x$ for all $x\in\check R$ and $\psi(g)=0$, 
then $g$ has a fixed point in $\check R$.   

Let now $g\in\check G^+$ and write
\bqn
\zeta(g^nx)-\zeta(x)=\sum_{i=0}^{n-1}[\zeta(g^{i+1}(x))-\zeta(g^i(x))]\,.
\eqn
Observe that since $g\succeq e$, all summands are non-negative.  
If now $\psi(g)=0$, then there exists a subsequence $(i_n)_{n\geq1}$ with
\bqn
\lim_{n\to\infty}\zeta(g^{i_n+1}(x))-\zeta(g^{i_n}(x))=0\,.
\eqn
Let $\Ff\subset\check R$ be a relatively compact fundamental domain for the $\langle Z\rangle$-action 
on $\check R$.  Then $g^{i_n}(x)=Z^{\ell_n}(y_n)$ for some $y_n\in\Ff$.
By taking into account \eqref{eq:q-m}, we deduce that 
\bqn
\lim_{n\to\infty}\zeta(g(y_n))-\zeta(y_n)=0\,.
\eqn
Since $g\succeq e$, we deduce from \eqref{eq:closed} that $\lim_{n\to\infty}d(gy_n,y_n)=0$
and hence any accumulation point of the sequence $(y_n)_{n\geq1}$ provides a fixed point for $g$.
\end{proof}

Finally observe that Theorem~\ref{thm:thm1.9} in the introduction follows from Proposition~\ref{prop:sandwich}(3),
Corollary~\ref{QMUnique} and Proposition~\ref{lem:dominant}.

\vskip1cm
\bibliographystyle{amsplain}

\def\cprime{$'$} \def\cprime{$'$}
\providecommand{\bysame}{\leavevmode\hbox to3em{\hrulefill}\thinspace}
\providecommand{\MR}{\relax\ifhmode\unskip\space\fi MR }
\providecommand{\MRhref}[2]{%
  \href{http://www.ams.org/mathscinet-getitem?mr=#1}{#2}
}
\providecommand{\href}[2]{#2}

\vskip1cm

\end{document}